\documentclass[12pt, reqno]{amsart}
\usepackage[utf8]{inputenc}
\usepackage[margin=2.5cm]{geometry}
\usepackage{graphicx}
\usepackage{amsmath}
\usepackage{amsthm}
\usepackage{amssymb,bbm}%
\usepackage[numbers, square]{natbib}
\usepackage{makecell}
\usepackage[hidelinks]{hyperref}

\usepackage{tikz}
\usepackage{subcaption}

\theoremstyle{plain}
\newtheorem{theorem}{Theorem}
\newtheorem{lemma}[theorem]{Lemma}
\newtheorem{corollary}[theorem]{Corollary}

\theoremstyle{definition}

\theoremstyle{remark}
\newtheorem{remark}[theorem]{Remark}

\makeindex
\makeglossary

\newcommand{\E}{\mathbb E\,}
\newcommand{\R}{\mathbb{R}}
\newcommand{\N}{\mathbb{N}}

\newcommand{\supp}{\mathop{\mathrm{supp}}\nolimits}

\newcommand{\conv}{{\mathrm{conv}}}

\newcommand{\diam}{\mathop{\mathrm{diam}}\nolimits}

\newcommand{\ind}{\mathbbm{1}}

\newcommand{\bx}{\mathbf{x}}
\newcommand{\by}{\mathbf{y}}

\newcommand{\bu}{\mathbf{u}}

\newcommand{\dd}{{\rm d}}

\newcommand{\revision}[1]{#1}

\usepackage{xcolor}

\begin{document}

\author[G.~Bonnet]{Gilles Bonnet}
\address{Gilles Bonnet, Faculty of Mathematics, Ruhr University Bochum, Germany}
\email{gilles.bonnet@rub.de}

\author[A.~Gusakova]{Anna Gusakova}
\address{Anna Gusakova, Faculty of Mathematics, Ruhr University Bochum, Germany}
\email{anna.gusakova@rub.de}

\author[C.~Th\"ale]{Christoph Th\"ale}
\address{Christoph Th\"ale, Faculty of Mathematics, Ruhr University Bochum, Germany}
\email{christoph.thaele@rub.de}

\author[D.~Zaporozhets]{Dmitry Zaporozhets}
\address{Dmitry Zaporozhets, St.~Petersburg Department of Steklov Institute of Mathematics, Russia}
\email{zap1979@gmail.com}

\title[Sharp inequalities for the mean distance]{Sharp inequalities for the mean distance\\ of random points in convex bodies}
\keywords{Concave funcion, convex body, geometric extremum problem, geometric inequality, intrinsic volume, mean distance, Riesz rearrangement inequality, sharp geometric inequality}
\subjclass[2010]{52A22, 52A40, 53C65, 60D05}
\thanks{The work of GB and AG was partially supported by the Deutsche Forschungsgemeinschaft (DFG) via RTG 2131 \textit{High-Dimensional Phenomena in Probability-Fluctuations and Discontinuity}.
The work of DZ was supported by the Foundation for the
Advancement of Theoretical Physics and Mathematics ``BASIS'' and by RFBR and DFG according to the research
project 20-51-12004.
}

\begin{abstract}
For a convex body $K\subset\mathbb{R}^d$ the mean distance $\Delta(K)=\mathbb{E}|\revision{X_1-X_2}|$ is the expected Euclidean distance of two independent and uniformly distributed random points $\revision{X_1,X_2}\in K$. Optimal lower and upper bounds for ratio between $\Delta(K)$ and the first intrinsic volume $V_1(K)$ of $K$ (normalized mean width) are derived and degenerate extremal cases are discussed. The argument relies on Riesz's rearrangement inequality and the solution of an optimization problem for powers of concave functions. The relation with results known from the existing literature is reviewed in detail.
\end{abstract}

\maketitle

\section{Introduction}

\begin{table}[t]
    \centering
    \begin{tabular}{|c||c|c|c|c|c|c|}
    \hline
    \parbox[0pt][2em][c]{0cm}{} $K$ & triangle & \makecell{square\\ (parallelogram)} & \makecell{regular \\ pentagon} & \makecell{regular\\ hexagon} & \makecell{regular \\ octagon} & \makecell{circle \\ (ellipse)}\\
    \hline
    \hline
    \parbox[0pt][2em][c]{0cm}{} $p(4,K)$ & $\frac{1}{3}$ & $\frac{11}{36}$ & $\frac{1}{45}(9+2\sqrt{5})$ & $\frac{298}{972}$ & $\frac{97+52\sqrt{2}}{576}$ & $\frac{35}{12\pi^2}$\\
    \hline
    \end{tabular}
     \caption{Particular values for the probability $p(4,K)$, see \cite{alikoski1938sylvestersche} and \cite[p.\ 114]{Solomon}.}
    \label{tab:Sylvester}
\end{table}

One of the most classical questions in \revision{the} area of geometric probability is Sylvester's question~\cite{jS64}, which asks for the probability $p(4,K)$ that the convex hull $\conv(X_1,X_2,X_3,X_3)$ of four independently and uniformly  distributed random points $X_1, X_2$, $X_3, X_4$ in a planar compact convex set $K\subset\R^2$ is a triangle. For particular sets $K$ the precise value of $p(4,K)$ is known and we refer to \cite[Sections 2.31--2.34]{KenMor}, \cite{rP89} and \cite[Chapter 5]{Solomon} for an extensive discussion. We also collect some examples in Table \ref{tab:Sylvester}. Using symmetrization arguments, Blaschke~\cite{wB17} was able to prove that for any compact convex set with non-empty interior $K\subset\mathbb{R}^2$  the two-sided inequality
\begin{align}\label{2328}
    \frac{35}{12\pi^2}\leq p(4,K) \leq\frac{1}{3}
\end{align}
holds. A glance at Table \ref{tab:Sylvester} shows that the lower bound is achieved if (and, in fact, only if) $K$ is an ellipse, and the upper bound if (and, in fact, only if) $K$ is a triangle. In this context one should note that $p(4,K)$ is invariant under affine transformations in the plane, which implies that the precise form of the ellipse and triangle does not play a role. It is not hard to verify that
\begin{align*}
    p(4,K)=4\,\frac{\E \mathrm A(\conv(X_1,X_2,X_3))}{\mathrm A(K)},
\end{align*}
where $\mathrm A(K)$ stands for the area of $K$, see \cite[Equation (8.11)]{SW08}. Therefore, Blaschke's inequality \eqref{2328} is equivalent to
\begin{align}\label{eq:Blaschke}
    \frac{35}{48\pi^2}\leq\frac{\E \mathrm A(\conv(X_1,X_2,X_3))}{\mathrm A(K)}\leq\frac{1}{12},
\end{align}
which gives the optimal lower and upper bound for the normalized mean area of the random triangle with vertices uniformly distributed in a planar compact convex set.

In the present paper we take up this classical and celebrated topic and instead of three points consider the situation where only two random points uniformly distributed in a compact convex set with non-empty interior $K\subset\mathbb{R}^2$ are selected. In this case, their convex hull is a random segment having a random length. It is thus natural to ask for the optimal bounds of the normalized average length of this segment. While the area of the random triangle is normalized by the area of $K$, the length of the random segment should be normalized by the perimeter of $K$ denoted by $P(K)$. In this paper we will prove that for any compact convex set $K\subset\R^2$ with non-empty interior the inequality
\begin{align}\label{1316}
    \frac{7}{60}<\frac{\E |X_1-X_2|}{\mathrm P(K)}<\frac{1}{6}
\end{align}
holds, where $|X_1-X_2|$ denotes the Euclidean distance of $X_1$ and $X_2$. We emphasize that in contrast to \eqref{eq:Blaschke} the inequalities on both sides of \eqref{1316} are strict, and we shall argue that \eqref{1316} is in fact optimal. Moreover, it will turn out that both bounds cannot be achieved by planar compact convex sets with interior points. In fact, the extremal cases correspond to two different degenerate situations, which we will described in detail. We would like to stress at this point that this surprising degeneracy phenomenon \revision{has only rarely been observed in similar situations so far in the existing literature around convex geometric inequalities. As such an exception we mention the inequalities for angle sums of convex polytopes by Perles and Shephard \cite{PerShep69}.}

Remarkably, we will be able to derive the analogue of \eqref{1316} in any dimension $d\geq 2$, where instead of the perimeter one normalizes the mean distance by the so-called first intrinsic volume of $K$, which in turn is a constant multiple of the mean width. We emphasize that this is in sharp contrast to Blaschke's inequality \eqref{eq:Blaschke} for which only a lower bound is known in any space dimension. This is the context of Busemann's random simplex inequality for which we refer to \cite{Busemann} or \cite[Theorem 8.6.1]{SW08} (according to results of Groemer \cite{Groemer0,Groemer} this holds more generally for convex hulls generated by an arbitrary number $n\geq d+1$ of random points and also for higher moments of the volume). A corresponding upper bound is still unknown, but in view of the planar case, it seems natural to expect that a sharp upper bound is provided by $d$-dimensional simplices. This is known as the simplex conjecture in convex geometric analysis and a positive solution would imply the famous hyperplane conjecture, see \cite{MilPaj} or \cite[Corollary 3.5.8]{BGVV2014}.

\medspace 

The remaining parts of this paper are structured as follows. In Section \ref{sec:History} we start with some historical remarks of what is known about the so-called mean distance of convex bodies. Our main result is presented in Section \ref{sec:Main}. Its proof is divided into several parts: proof of the lower bound (Section \ref{subsec:LowerBound}), proof of the upper bound (Sections \ref{subsec:UpperBound1}--\ref{subsec:UpperBound3}) and sharpness of the estimates (Section \ref{1305}).

\begin{figure}[t]
    \centering
    \includegraphics[width=0.3\columnwidth]{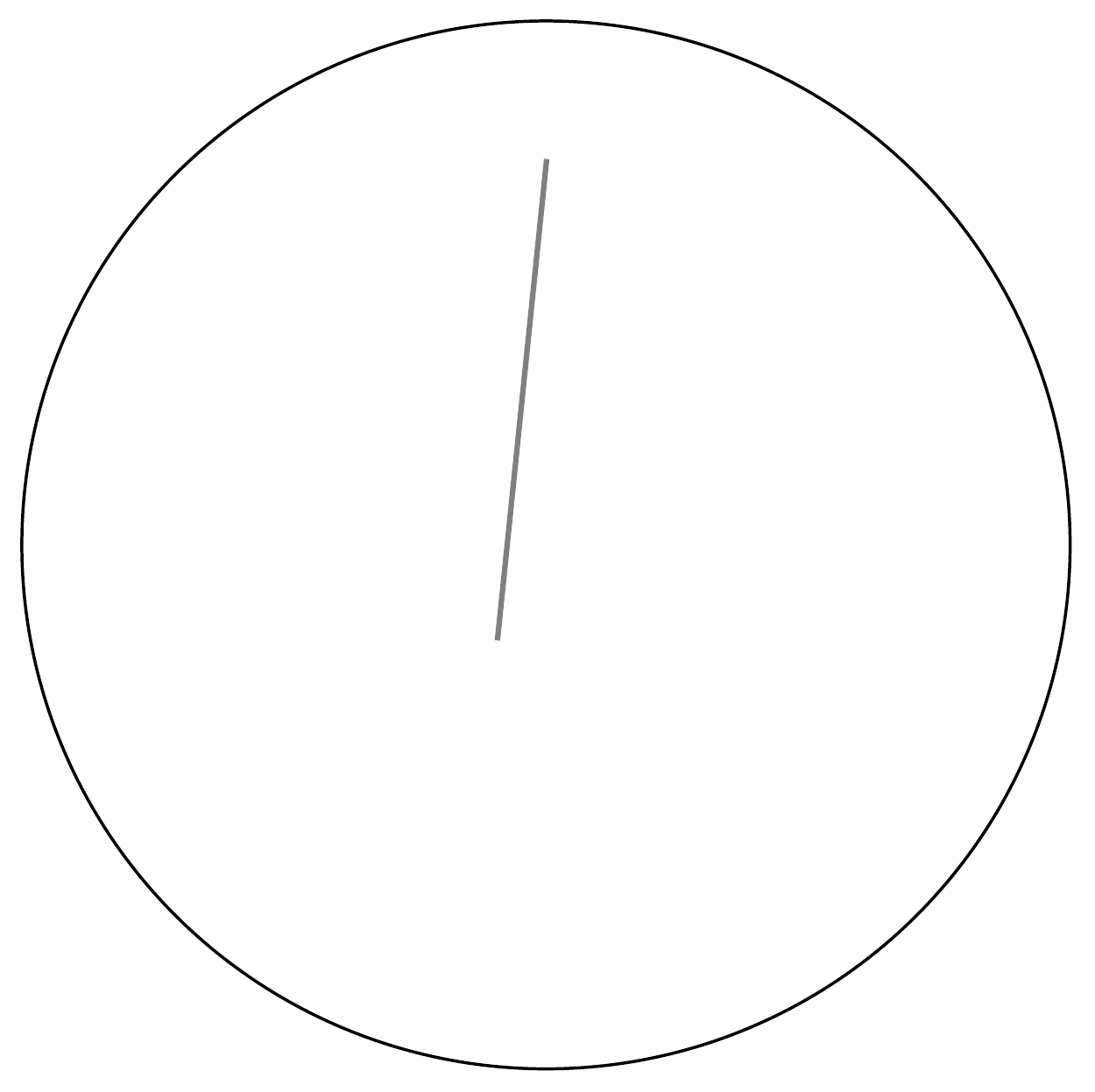}\quad
    \includegraphics[width=0.3\columnwidth]{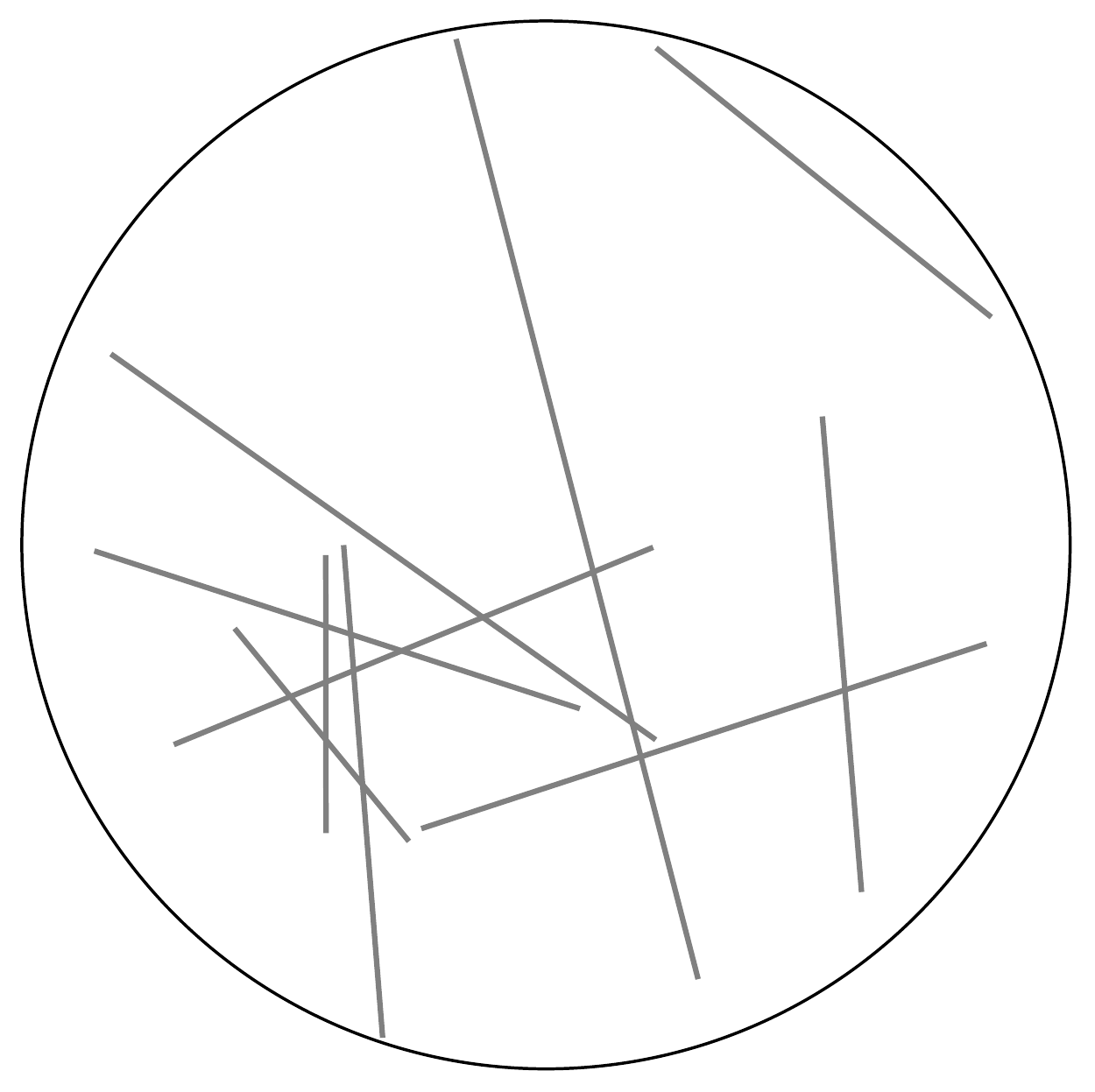}\quad
    \includegraphics[width=0.3\columnwidth]{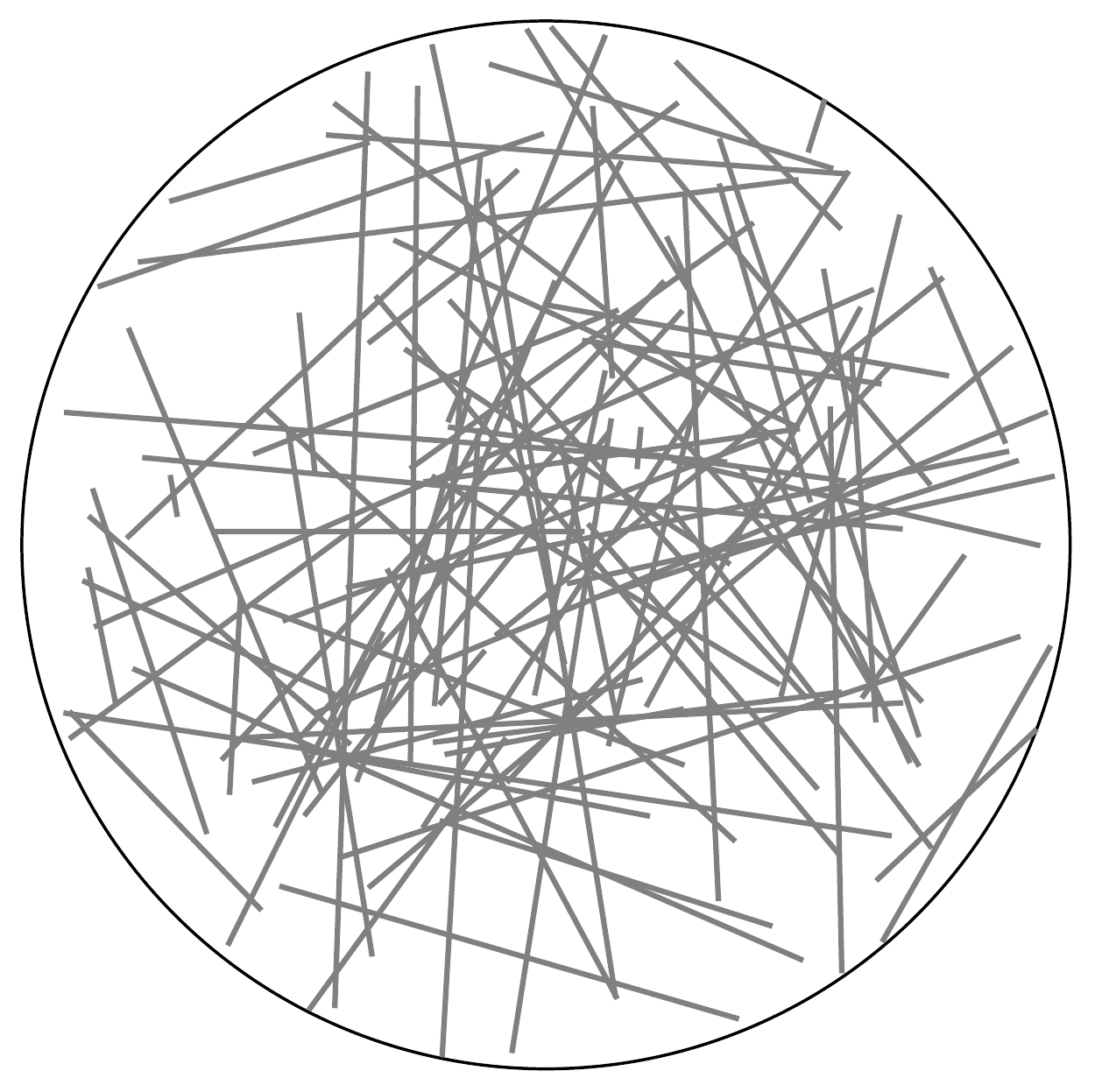}
    \caption{Simulation of $1$ (left), $10$ (middle) and $100$ (right) random segments generated by two random points $X_1,X_2$ in a planar disc $B^2$ of radius $1$. The mean distance $\Delta(B^2)=\mathbb{E}|X_1-X_2|=\frac{128}{45\pi}\approx 0.9054\ldots$ is the expected length of these segments.}
    \label{fig_Segment}
\end{figure}

\section{Historical remarks}\label{sec:History}

Before presenting our main results, we start with some historical remarks, which should help reader to bring our results in line with what is known from the literature. We also introduce some basic notation that will be used throughout the paper.

By a convex body in $\R^d$ we understand a compact convex subset of $\R^d$ with non-empty interior.  Let $K\subset\R^d$ be a convex body and let \revision{$X_1$ and $X_2$} be two independent random vectors uniformly distributed in $K$. We will denote by $\Delta(K)$ the mean distance between \revision{$X_1$ and $X_2$}, that is,
\begin{align*}
    \Delta(K):=\E |\revision{X_1-X_2}|=\frac{1}{|K|^2}\int_K\int_K|\revision{\bx_1-\bx_2}|\,\dd \revision{\bx_1 \,\dd \bx_2}.
\end{align*}
Here and in what follows, $|A|$ will denote the volume of a measurable set $A\subset\R^d$ of the appropriate dimension, by which we understand the Lebesgue measure with respect to the affine hull of $A$. It is known from \cite[Equation~(21)]{gC67} or \cite[Equation~(34)]{jK69} that for any $p>-d$,
\begin{align*}
    \int_K&\int_K|\bx-\by|^p\,\dd \bx \,\dd \by=\frac{2}{(d+p)(d+p+1)}\int_{\mathbb S^{d-1}}\int_{\bu^\perp}|\bx_\bu\cap K|^{d+p+1}\, \lambda_{\bu^\perp}(\mathrm d\bx)\mu(\mathrm d\bu),
\end{align*}
where $\bx_\bu$ is a line through $\bx$ parallel to $\bu$, \revision{$\bu^\perp$ is the linear hyperplane
orthogonal to the unit vector u}, and $\lambda_{\bu^\perp}$ and $\mu$ are the Lebesgue measures on $\bu^\perp$ and the spherical Lebesgue measure on $\mathbb S^{d-1}$, respectively. Therefore, an alternative form for the mean distance $\Delta(K)$ is given by
\begin{align}\label{1324}
    \Delta(K)=\frac{2}{(d+1)(d+2)|K|^2}\int_{\mathbb S^{d-1}}\int_{\bu^\perp}|\bx_\bu\cap K|^{d+2}\, \lambda_{\bu^\perp}(\revision{\mathrm d}\mathrm \bx)\mu(\revision{\mathrm d}\mathrm \bu).
\end{align}

There are relatively few examples of convex bodies $K$ for which the exact value of $\Delta(K)$ is actually known. Most of them are $2$-dimensional and all depend on the explicit shape of $K$. The simplest ones are the circle ($2$-dimensional ball) $B^2(r)$ of radius $r>0$  and the  regular triangle $T^2(a)$ with side length $a>0$. For these sets we have
\begin{align*}
    \Delta(B^2(r))=\frac{128}{45\pi}r\qquad\text{and}\qquad \Delta(T^2(a))=a\bigg(\frac15+\frac3{20}\log3\bigg)
\end{align*}
from~\cite{dF64} and~\cite[Page~785]{mC85}; see Figure \ref{fig_Segment} for an illustration of the first case. However, even for a rectangle $R(a,b)$ with side lengths $0<a\leq b$ the formula becomes much more involved. In fact, from \cite{bG51} it is known that
\begin{align*}
    \Delta(R(a,b))=\frac{1}{15}\bigg[&\frac{a^3}{b^2}+\sqrt{a^2+b^2}\bigg(3-\frac{a^2}{b^2}-\frac{b^2}{a^2}\bigg)
    \\&+\frac52\bigg(\frac{b^2}{a}+\log\frac{a+\sqrt{a^2+b^2}}{b}+\log\frac{b+\sqrt{a^2+b^2}}{a}\bigg)\bigg].
\end{align*}
For the cases when $K$ is an arbitrary triangle, ellipse or parallelogram we refer to~\cite{tS85}. The mean distance for a regular hexagon $H(a)$ with side length $a>0$ was considered in~\cite{ZP}. In this case 
\begin{align}\label{eq:Hexagon}
    \revision{\Delta(H(a))=a\bigg[{7\sqrt{3}\over 30}-{7\over 90}+{1\over 60}\Big[28\log(2\sqrt{3}+3)+29\log(2\sqrt{3}-3)\Big]\bigg].}
\end{align}
\revision{The distribution function of $|X_1-X_2|$ was calculated for an arbitrary regular polygon in~\cite{uB14}, while moments (especially of order one, two and four) are the content of the recent article \cite{Baesel21}.}

In higher dimensions, the number of examples for which an exact formula for $\Delta(K)$ is available is rather limited. Perhaps the most well-known one is the so-called Robbins constant, which gives $\Delta([0,1]^3)$ for the $3$-dimensional unit cube:
\begin{align*}
    \Delta([0,1]^3)=\frac1{105}\bigg[4&+17\sqrt2-6\sqrt3
    +21\log(1+\sqrt 2)+42\log(2+\sqrt 3)-7\pi\bigg].
\end{align*}
For the multidimensional unit cube $[0,1]^d$ with $d\geq 2$, $\Delta([0,1)^d)$ is known as a box integral, which does not have a closed form expression for dimensions $d\geq 4$, see~\cite{BBC07}.

A non-trivial case for which the answer is known in any dimension is the unit $d$-dimensional ball $B^d=B^d(1)$. In fact, a special case of~\cite[Theorem~2]{rM71} yields that
\begin{align}\label{2312}
    \Delta(B^d)=\frac{2^{2d+2}d\cdot \big[\Gamma\big(\frac d2+1\big)\big]^2}{(2d+1)!!(d+1)\pi},
\end{align}
where $\Gamma(\,\cdot\,)$ is the Gamma function and $d!!=d(d-2)\cdots$ the double factorial. If the convex body $K$ is an ellipsoid in $\R^d$ with semi-axes $a_1,\dots,a_d>0$, then~\eqref{2312} can be generalized as follows (see~\cite[Theorem~3.1]{lH14} combined with~\eqref{1324}):
\begin{align*}
    \Delta(K)={\frac{2^{d+1}\big[\Gamma\big(\frac d2+1\big)\big]^3}{(d+1)\pi^{(d+1)/2}\Gamma\big(d+\frac{3}{2}\big)}}\int_{\mathbb S^{d-1}}\sqrt{a_1^2u_1^2+\ldots+a_d^2u_d^2}\,\,\mu(\dd\bu).
\end{align*}

Apart from the exact formulas we presented so far, there are several  bounds for $\Delta(K)$ in terms of different geometric characteristics of the convex body $K\subset\R^d$. The most well-known one relates $\Delta(K)$ with the volume $|K|$ of $K$. It says that
\begin{align*}
    \Delta(K)\geq\frac{2^{2d+2}d\cdot \big[\Gamma\big(\frac d2+1\big)\big]^{2+1/d}}{(2d+1)!!(d+1)\pi^{3/2}}|K|^{1/d},
\end{align*}
and equality holds if and only if $K$ is a $d$-dimensional Euclidean ball. The result can be found in~\cite{wB18} for dimension $d=2$ and in~\cite{rP90} for higher dimensions $d\geq 3$.

In~\cite{BP09}, $\Delta(K)$ was bounded from above by the diameter $\diam(K)$ of $K$. The inequality says that
\begin{align}\label{eq:UpperBdDiam}
    \Delta(K)\leq\diam(K)\sqrt{\frac{2d}{\pi(d+1)}}\cdot\frac{2^{d-2}\left[\Gamma(d/2)\right]^2}{\Gamma(d-1/2)}
\end{align}
for any convex body $K\subset\R^d$. Apparently, this bound is far from being optimal. In the next section we will present a complementing best possible lower bound and an upper bound, which improves \eqref{eq:UpperBdDiam} in low dimensions, see Corollary \ref{cor:Diam}.

To the best of our knowledge, bounds in terms of other characteristics (not following from the existed ones) are not known. 

\section{Main result}\label{sec:Main}

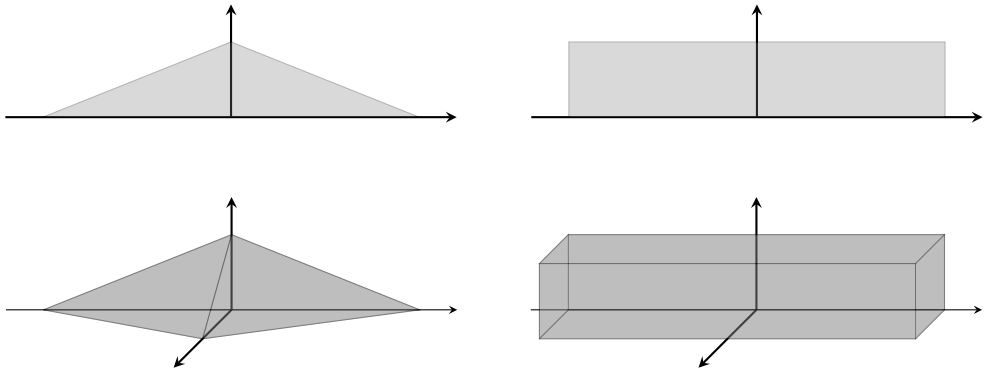
\begin{figure}[t]
\begin{center}
\begin{tikzpicture}
\draw[thick,-stealth] (-3,0)--(3,0); 
\draw[thick,-stealth] (0,0)--(0,1.5); \coordinate (A) at (-2.5,0);
\coordinate (B) at (2.5,0);
\coordinate (C) at (0,1);
\draw[fill=gray,opacity=0.3] (A)--(B)--(C)--(A);
\end{tikzpicture}
\qquad
\begin{tikzpicture}
\draw[thick,-stealth] (-3,0)--(3,0); 
\draw[thick,-stealth] (0,0)--(0,1.5); \coordinate (A) at (-2.5,0);
\coordinate (B) at (2.5,0);
\coordinate (C) at (2.5,1);
\coordinate (D) at (-2.5,1);
\draw[fill=gray,opacity=0.3] (A)--(B)--(C)--(D)--(A);
\end{tikzpicture}\\
\vspace{1cm}
\begin{tikzpicture}
\draw[thick,-stealth] (0,0,0)--(0,0,2); 
\draw[thick,-stealth] (0,0,0)--(0,1.5,0);
\draw[-stealth] (-3,0,0)--(3,0,0);

\coordinate (A) at (-2.5,0,0);
\coordinate (B) at (2.5,0,0);
\coordinate (C) at (0,0,1);
\coordinate (D) at (0,1,0);
\draw[fill=gray,opacity=0.3] (A)--(B)--(C)--(A);
\draw[fill=gray,opacity=0.3] (A)--(B)--(D)--(A);
\draw[fill=gray,opacity=0.3] (A)--(C)--(D)--(A);
\draw[fill=gray,opacity=0.3] (B)--(C)--(D)--(B);
\end{tikzpicture}
\qquad
\begin{tikzpicture}
\draw[thick,-stealth] (0,0,0)--(0,0,2); 
\draw[thick,-stealth] (0,0,0)--(0,1.5,0);
\draw[-stealth] (-3,0,0)--(3,0,0);

\coordinate (A1) at (-2.5,0,0);
\coordinate (A2) at (-2.5,1,0);
\coordinate (A3) at (-2.5,1,1);
\coordinate (A4) at (-2.5,0,1);
\coordinate (B1) at (2.5,0,0);
\coordinate (B2) at (2.5,1,0);
\coordinate (B3) at (2.5,1,1);
\coordinate (B4) at (2.5,0,1);
\draw[fill=gray,opacity=0.3] (A1)--(A2)--(A3)--(A4)--(A1);
\draw[fill=gray,opacity=0.3] (B1)--(B2)--(B3)--(B4)--(B1);
\draw[fill=gray,opacity=0.3] (A1)--(B1)--(B2)--(A2)--(A1);
\draw[fill=gray,opacity=0.3] (A2)--(B2)--(B3)--(A3)--(A2);
\draw[fill=gray,opacity=0.3] (A3)--(B3)--(B4)--(A4)--(A3);
\draw[fill=gray,opacity=0.3] (A4)--(B4)--(B1)--(A1)--(A4);
\end{tikzpicture}
\end{center}
    \caption{Illustration of $K_\delta$ (left) and $K_\delta'$ (right) for $d=2$ (upper row) and $d=3$ (lower row).}
    \label{fig:Extremizers}
\end{figure}

Let $K\subset\R^d$ be a convex body. The main goal of this paper is to derive the optimal lower and upper bounds for $\Delta(K)$ normalized by the mean width of $K$, which is given by
\begin{align*}
    W(K):=\int_{\mathbb S^{d-1}}|P_\bu K|\,\mu(\dd\bu),
\end{align*}
where $|P_\bu K|$ denotes the length of the projection of $K$ onto the line spanned by $\bu$.

An obstacle when working with the mean width is its dependence on the dimension of the ambient space. In fact, if we embed $K$ into $\R^n$ with $n\geq d$, then $W(K)$ is strictly decreasing with respect to $n$. That is why it is convenient to use the following normalized version of the mean width:
\begin{align}\label{1339}
    V_1(K):=\sqrt\pi\frac{\Gamma\big(\frac{d+1}{2}\big)}{\Gamma\big(\frac{d}{2}\big)}\int_{\mathbb S^{d-1}}|P_\bu K|\,\mu(\dd\bu).
\end{align}
This quantity is known as {the first intrinsic volume} of $K$, and it  does not depend on the dimension of the ambient space. In particular, this property \revision{and  \eqref{1339} with $d=1$ imply} that for any one-dimensional line segment $I\subset\R^d$, $V_1(I)$ coincides with the length $|I|$ of $I$, i.e., 
\begin{align}\label{1334}
    V_1(I)=|I|.
\end{align}

Now we are ready to formulate our main result, whose proof is postponed to the next section. Denote by $e_1,\dots,e_d$ the standard orthonormal basis in $\R^d$.

\begin{theorem}\label{2311}
For any convex body $K\subset\R^d$ one has that
\begin{align}\label{1719}
    \frac{3d+1}{2(d+1)(2d+1)}<\frac{\Delta(K)}{V_1(K)}<\frac13.
\end{align}
Moreover, this inequality is sharp in the following sense: the two families of the convex bodies defined, for $\delta>0$, as
\begin{align*}
    K_\delta &:=\conv({e_1,-e_1,\delta e_2,\delta e_3},\dots,\delta e_d),\\
    K'_\delta &:=[{-1},1]\times[0, \delta]^{d-1}
\end{align*}
satisfy
\begin{align}\label{1308}
    \lim_{\delta\to 0}\frac{\Delta(K_\delta)}{V_1(K_\delta)}=\frac{3d+1}{2(d+1)(2d+1)}\quad\text{and}\quad\lim_{\delta\to 0}\frac{\Delta(K'_\delta)}{V_1(K'_\delta)}=\frac13.
\end{align}
\end{theorem}
\begin{remark}
    The constructed families $K_\delta, K_\delta'$ and~\eqref{1308} show that on the space of compact sets, $\Delta(\,\cdot\,)$ is neither continuous in Hausdorff metric nor monotone with respect to set inclusion. The two sets $K_\delta$ and $K_\delta'$ are illustrated for $d=3$ in Figure \ref{fig:Extremizers}.
\end{remark}

Let us derive the following consequence of Theorem \ref{2311}, which yields an optimal lower bound for the quantity $\Delta(K)/\diam(K)$, which was already discussed in the previous section. 
\begin{corollary}\label{cor:Diam}
For any convex body $K\subset\R^d$ one has that
\begin{align}\label{1119}
    \frac{3d+1}{2(d+1)(2d+1)}<\frac{\Delta(K)}{\diam(K)}<{\sqrt{\pi}\over 3}\,\frac{\Gamma\big(\frac{d+1}{2}\big)}{\Gamma\big(\frac{d}{2}\big)}.
\end{align}
Moreover, this inequality from below is sharp in the following sense: for $K_\delta$, $\delta>0$, as defined in Theorem~\ref{2311} we have that
\begin{align}\label{1706}
     \lim_{\delta\to 0}\frac{\Delta(K_\delta)}{\diam(K_\delta)}=\frac{3d+1}{2(d+1)(2d+1)}.
\end{align}
\end{corollary}
\begin{remark}
    It can also be checked numerically that the upper bound in \eqref{eq:UpperBdDiam} is worse than the right-hand side of~\eqref{1119} in dimensions $d=2$, $d=3$ and $d=4$, but is still better for $d\geq 5$. In fact, it should be observed that, as $d\to\infty$, the constant on the right hand side of \eqref{eq:UpperBdDiam} behaves like $1-{5\over 8d}+O(d^{-2})$, while our upper bound satisfies ${\sqrt{{\pi\over 18}\,d}}+O(d^{-1/2})$.
\end{remark}
\begin{proof}[Proof of Corollary~\ref{cor:Diam}]
For any convex body $K\subset\R^d$ it is possible to find an open interval $I\subset K$ satisfying $|I|=\diam(K)$. Next, we recall that it is well known that the intrinsic volumes are monotone with respect to set inclusion. Therefore,
\begin{align*}
    V_1(K)\geq V_1(I)=|I|=\diam(K),
\end{align*}
which together with Theorem~\ref{2311} implies the lower bound. The upper bound follows from the definition of $V_1(K)$ and the fact that the mean width $W(K)$ satisfies $W(K)\leq\diam(K)$. Indeed, this is a consequence of the observation that the maximal width $\max_{\bu\in\mathbb{S}^{d-1}}|P_{\bu}K|$ of $K$ coincides with the diameter of $K$. This together with~\eqref{1719} and~\eqref{1308} yields the first part of the corollary.

To prove the second one, note that
\begin{align*}
    \lim_{\delta\to 0}\diam(K_\delta)={2}.
\end{align*}
Moreover, since the intrinsic volumes are continuous in Hausdorff metric, we have 
\begin{align*}
    \lim_{\delta\to 0}V_1(K_\delta)=V_1([{-1},1])={2}.
\end{align*}
Combining the last two equations with the left-hand side of~\eqref{1308} gives~\eqref{1706} and completes the argument.
\end{proof}

\section{Proof of Theorem~\ref{2311}}

\subsection{Preliminaries}

Before presenting the proof of Theorem~\ref{2311} we start with some general comments on $\Delta(K)$. It follows from~\eqref{1334} and~\eqref{1339} along with Fubini's theorem that
\begin{align*}
    \E|X_1-X_2|&=\sqrt\pi\frac{\Gamma\big(\frac{d+1}{2}\big)}{\Gamma\big(\frac{d}{2}\big)}\,\E\int_{\mathbb S^{d-1}}|P_\bu X_1-P_\bu X_2|\,\mu(\mathrm d\bu)
    \\
    &=\sqrt\pi\frac{\Gamma\big(\frac{d+1}{2}\big)}{\Gamma\big(\frac{d}{2}\big)}\int_{\mathbb S^{d-1}}\E|P_\bu X_1-P_\bu X_2|\,\mu(\mathrm d\bu).
\end{align*}
Let us fix some $\bu\in\mathbb S^{d-1}$. Again by Fubini's theorem, we see that
\begin{align}\label{929}
    \E|P_\bu X_1-P_\bu X_2|&=\frac{1}{|K|^2}\int_K\int_K|P_\bu\bx_1-P_\bu\bx_2|\,\dd \bx_1 \,\dd \bx_2
    \\\notag
    &=\int_{\inf\limits_{\bx\in K}\langle \bx,\bu\rangle}^{\sup\limits_{\bx\in K}\langle \bx,\bu\rangle}\int_{\inf\limits_{\bx\in K}\langle \bx,\bu\rangle}^{\sup\limits_{\bx\in K}\langle \bx,\bu\rangle}|t_1-t_2|\, \tilde h(t_1)\, \tilde h(t_2)\,\dd t_2\, \dd t_1,
\end{align}
where
\begin{align}\label{eq:hTilde}
    \tilde h(t)=\tilde h_{K,u}(t):=\frac{|K\cap (t\bu+\bu^\perp)|}{|K|}.
\end{align}
Let $L:\R\to\R$ be an affine function which maps the interval $[-1,1]$ to $\big[\inf_{\bx\in K}\langle \bx,\bu\rangle,\sup_{\bx\in K}\langle \bx,\bu\rangle\big]$.
Clearly, the slope of $L$ equals 
\begin{align*}
    \frac{\sup_{\bx\in K}\langle \bx,\bu\rangle-\inf_{\bx\in K}\langle \bx,\bu\rangle}2=\frac{|P_\bu K|}2.
\end{align*}
Changing twice coordinates according to the transformation $L$ allows us in view of~\eqref{929} to conclude that
\begin{align*}
    \E|P_\bu X_1-P_\bu X_2|=\frac{|P_\bu K|}2\int_{-1}^1\int_{-1}^1|t_1-t_2|\, h(t_1)\, h(t_2)\,\dd t_2\, \dd t_1
\end{align*}
with $h(t)$ given by
\begin{align}\label{eq:h}
    h(t)=h_{K,u}(t):=\frac{|P_\bu K|}2\tilde h(Lt),
\end{align}
where $\tilde h$ is as in \eqref{eq:hTilde}.
Introducing the abbreviation
\begin{align}\label{eq:IntegralI(h)}
    I(h):=\frac12\int_{\R}\int_{\R}|t_1-t_2|\, h(t_1)\, h(t_2)\,\dd t_2\,\dd t_1,
\end{align}
we arrive at the identity
\begin{align}\label{2226}
    \E&|P_\bu X_1-P_\bu X_2|=|P_\bu K|\,I(h).
\end{align}

Next, we note that the function $h$ possesses the following four properties, where we write $\supp(h)$ for the support of $h$, \revision{the smallest closed set containing all points $t\in\mathbb{R}$ such that $h(t)\ne 0$}:
\begin{enumerate}
\item[(a)] $h\geq 0$;
\item[(b)] $\supp(h)= [-1,1]$;
\item[(c)] $\int_{\R} h(t)\,\dd t=1$;
\item[(d)] $h^{1/(d-1)}$ is concave on its support. 
\end{enumerate}
The first three properties are evident, while the last one is a direct consequence of Brunn's concavity principle (see, e.g.,~\cite[Theorem~2.3]{aK05}).

Our next two steps are to determine the infimum and supremum of $I(\,\cdot\,)$ over all functions satisfying properties (a)--(d). We will tackle both problems separately in Sections \ref{subsec:LowerBound} (lower bound) and \ref{subsec:UpperBound1}--\ref{subsec:UpperBound3} (upper bound). In Section~\ref{1305} we will show that~\eqref{1308} holds true.


\subsection{The lower bound}\label{subsec:LowerBound}

The crucial ingredient in the getting the lower bound is Riesz's rearrangement inequality. In our paper, we need its one-dimension version only.  To formulate it let us recall the definition of symmetric decreasing rearrangement. To this end, for any non-negative measurable function $f:\R\to\R_+$ and $\tau\geq 0$ denote by $E_\tau=E_{f,\tau}$ its excursion set
\begin{align*}
    E_\tau:=\{t\in\R:f(t)>\tau\}.
\end{align*}
It is straightforward to see that $f$ can be recovered from $E_\tau$:
\begin{align*}
    f(t)=\int_0^\infty\ind_{E_\tau}(t)\,\dd \tau.
\end{align*}
Assuming  that $|E_\tau|<\infty$ for any $\tau\geq0$, denote by $f^*$ the {symmetric decreasing rearrangement} $f^*$ of $f$, which is defined  as
\begin{align*}
    f^*(t):=\int_0^\infty\ind_{\left[-\frac{|E_\tau|}2,\frac{|E_\tau|}2\right]}(t)\,\dd \tau.
\end{align*}
In other words, $f^*$ is a unique even and on the positive half-line decreasing function, whose  level sets have the same measure as the level sets of $f$. Geometrically, the subgraph of $f^*$ is obtained from the subgraph of $f$ by \emph{Steiner symmetrization} with respect to the abscissa.
\begin{remark}\label{1129}
    Since Steiner symmetrization preserves convexity (see, e.g.,~\cite[Proposition~7.1.7]{KP99}) \revision{and a function is concave if and only if its subgraph is convex}, it follows that $f^*$ is concave given $f$ is concave.
\end{remark}
 
The Riesz rearrangement inequality (see, e.g.,~\cite[Section~3.6]{LL01}) states that for any non-negative measurable functions $f_1,f_2,g:\R\to\R_+$ 
with level sets of finite measure we have
 \begin{equation}\label{eq:Riesz}
\begin{split}
&\int_{\R}\int_{\R} f_1(t_1)g(t_1-t_2)f_2(t_2)\,\dd t_1\,\dd t_2\leq \int_{\R}\int_{\R} f_1^*(t_1)g^*(t_1-t_2)f_2^*(t_2)\,\dd t_1\,\dd t_2.
\end{split}
 \end{equation}
 Now, let us take
 \begin{align*}
     g(t)=\max(0,2-|t|)\quad\text{and}\quad f_1=f_2=h
 \end{align*}
 with $h$ given by \eqref{eq:h}.
It is easy to check that with $h$ also $h^*$ satisfies properties~(a)--(d)  listed in the previous section: indeed,~(a)--(c) are due to the basic properties of the function rearrangement, see~\cite[Section~3.3]{LL01}. To show~(d), first note that
\begin{align*}
    \big(h^{1/(d-1)}\big)^*=(h^*)^{1/(d-1)};
\end{align*}
for the proof see~\cite[Section~3.3, Property~(v)]{LL01}. Now~(d) follows from Remark~\ref{1129}.

Clearly, $g^*=g$. Therefore from  property~(b) of $h$ and $h^*$ it follows that
\begin{align*}
    h(t_1)g(t_1-t_2)h(t_2)=(2-|t_1-t_2|) h(t_1)h(t_2)
\end{align*}
and 
\begin{align*}
    h^*(t_1)g^*(t_1-t_2)h^*(t_2)=(2-|t_1-t_2|) h^*(t_1)h^*(t_2).
\end{align*}
Applying Riesz’s inequality \eqref{eq:Riesz} and noting that by property~(c),
\begin{align*}
    \int_{\R}\int_{\R} 2h(t_1)h(t_2)\,\dd t_1\,\dd t_2= \int_{\R}\int_{\R}2 h^*(t_1)h^*(t_2)\,\dd t_1\,\dd t_2=2,
\end{align*}
we conclude that
\begin{align*}
    I(h)\geq I(h^*),
\end{align*}
where we recall that $I(h)$ is given by \eqref{eq:IntegralI(h)}. Thus, from this moment on we can and will assume that $h$ is an even function. We will also use the notation
\begin{equation}\label{eq:HTilde}
    \tilde H(t):=\int_0^{t}h(s)\,\dd s
\end{equation}
in what follows.

\begin{lemma}\label{lem:LB1}
Let $h:\R\to\R$ be an even function satisfying (a)-\revision{(d)}. Then $I(h)=\frac12-2\int_0^1\tilde H^2(t)\,\dd t$.
\end{lemma}
\begin{proof}
We start by noting that
\begin{align*}
    I(h)=\int_{\R}\int_{-\infty}^{t_1}(t_1-t_2)\, h(t_1)\, h(t_2)\,\dd t_2 \,\dd t_1.
\end{align*}
Integration-by-parts thus leads to
\begin{align*}
    \int_{-\infty}^{t_1}t_1\, h(t_1)\, h(t_2)\,\dd t_2=t_1\,h(t_1)\,H(t_1)
\end{align*}
and
\begin{align*}
    \int_{-\infty}^{t_1}t_2\, h(t_1)\, h(t_2)\,\dd t_2=h(t_1)\,\bigg(t_1\,H(t_1)-\int_{-\infty}^{t_1}H(t_2)\dd t_2\bigg),
\end{align*}
where we put
$
    H(t):=\int_{-\infty}^{t}h(s)\,\dd s.
$
As a consequence,
\begin{align*}
    I(h)=\int_{\R}h(t_1)\int_{-\infty}^{t_1}H(t_2)\,\dd t_2\,\dd t_1.
\end{align*}
Again applying integration-by-parts and property (c) in the first and property~(b) in the last step gives
\begin{align*}
    I(h)&=
    \int_{\R}H(t_2)\,\dd t_2-\int_{\R}H(t_1)\,H(t_1)\,\dd t_1
    \\
    &=\int_{\R} H(t)(1-H(t))\,\dd t\\
    &=\int_{-1}^1 H(t)(1-H(t))\,\dd t.
\end{align*}
Since $h$ is even, we have
\begin{align*}
    H(0)=\frac{1}{2}\qquad\text{and}\qquad H(-t)=1-H(t).
\end{align*}
Therefore, recalling the definition of $\tilde H(t)$
we see that
\begin{equation}\label{1335}
\begin{split}
    I(h)&=2\int_0^1 H(t)(1-H(t))\,\dd t
    \\
    &=2\int_0^1\bigg(\frac12+ \tilde H(t)\bigg)\bigg(\frac12- \tilde H(t)\bigg)\,\dd t\\
    &=\frac12-2\int_0^1\tilde H^2(t)\,\dd t.
\end{split}    
\end{equation}
The argument is thus complete.
\end{proof}

Next, we consider the function
\begin{align}\label{2217}
    h_0(t):=
    \begin{cases}
        \frac d2\left(1-|t|\right)^{d-1}&:|t|\leq1,
        \\
        0 &:|t|\geq1.
    \end{cases}
\end{align}

\begin{lemma}\label{lem:LB2}
The function $h_0$ satisfies properties (a)-(d) and $I(h_0)=\frac{3d+1}{2(d+1)(2d+1)}$.
\end{lemma}
\begin{proof}
It is straightforward that $h_0$ possesses properties~(a)--(d). To compute $I(h_0)$ we put
\begin{align*}
    \tilde H_0(t):=\int_{0}^{t}h_0(s)\,\dd s.
\end{align*}
Using the substitution $1-s=u$, we see that
\begin{align*}
\tilde H_0(t) = {d\over 2}\int_0^t(1-s)^{d-1}\,\dd s = {d\over 2}\int_{1-t}^1 u^{d-1}\,\dd u = {1\over 2}\big(1-(1-t)^d\big).
\end{align*}
As a consequence, applying the substitutions $u=1-t$ and $u^d=v$ we see that
\begin{align*}
\int_0^1\tilde H_0^2(t)\,\dd t &= {1\over 4}\int_0^1\big(1-(1-t)^d\big)^2\,\dd t = {1\over 4}\int_0^1(1-u^d)^2\,\dd u\\
&={1\over 4d}\int_0^1(1-v)^2v^{{1\over d}-1}\,\dd v.
\end{align*}
The last integral is known as the Euler Beta function $B(1/d,3)$ and thus simplifies to
\begin{align*}
{1\over 4d}{\Gamma({1\over d})\Gamma(3)\over\Gamma(3+{1\over d})} = {1\over 2d}{\Gamma({1\over d})\over\big(2+{1\over d}\big)\big(1+{1\over d}\big){1\over d}\Gamma({1\over d})} = {d^2\over 2(2d+1)(d+1)}.
\end{align*}
As a consequence, using Lemma \ref{lem:LB1} we find that
\begin{align}\label{2218}
    I(h_0)={1\over 2}-2\int_0^1\tilde H_0^2(t)\,\dd t=\frac{3d+1}{2(d+1)(2d+1)}.
\end{align}
This completes the proof of the lemma.
\end{proof}

Out task is now to show that for any even function $h:\R\to\R$ which satisfies properties (a)--(d) and is different from $h_0$ we have that $I(h)>I(h_0)$.

\begin{lemma}\label{lem:LB3}
Let $h:\R\to\R$ be an even function satisfying (a)-(d). \revision{If $h$ differs from $h_0$ on a set of positive Lebesgue measure, then} $I(h)>I(h_0)$.
\end{lemma}
\begin{proof}
We start by noting that proving $I(h)>I(h_0)$ is in view of~\eqref{1335} equivalent to proving that
\begin{align}\label{1315}
    \int_0^1\tilde H_0^2(t)\,\dd t>\int_0^1\tilde H^2(t)\,\dd t,
\end{align}
\revision{where $\tilde H_0$ is defined by $h_0$ the same way as $\tilde H$ by $h$ in~\eqref{eq:HTilde}.}
For that purpose, we represent the difference $h-h_0$ as
\begin{equation}\label{1306}
\begin{split}
    h-h_0&=\big(h^{1/(d-1)}-h^{1/(d-1)}_0\big)\bigg(\sum_{i=0}^{d-1}h^{i/(d-1)}h_0^{(d-1-i)/(d-1)}\bigg),
\end{split}    
\end{equation}
where 
\begin{align}\label{1307}
    \sum_{i=0}^{d-1}h^{i/(d-1)}h_0^{(d-1-i)/(d-1)}
\end{align}
is a positive function on $(-1,1)$.
By property~(d), $h^{1/(d-1)}$ is concave on $[0,1]$, while $h^{1/(d-1)}_0$ is linear (and hence convex) on this interval. Therefore,
$
    h^{1/(d-1)}-h_0^{1/(d-1)}
$
is a concave function on $[0,1]$.

Note that
$
    h^{1/(d-1)}(1)\geq 0=h^{1/(d-1)}_0(1),
$
and if $(h^{1/(d-1)}-h^{1/(d-1)}_0)(0)\ge 0$ we conclude that $h^{1/(d-1)}-h^{1/(d-1)}_0\ge 0$ on $[0,1]$ due to concavity. This would mean that $h-h_0\ge 0$ on $[0,1]$, which leads to contradiction for $h\neq h_0$ because of properties (a) and (c). Thus, it follows that there exists $t_0\in(0,1)$ such that 
\begin{align*}
    h^{1/(d-1)}-h_0^{1/(d-1)}\leq 0\quad\text{on}\quad [0,t_0]
\end{align*}
and
\begin{align*}
    h^{1/(d-1)}-h_0^{1/(d-1)}\geq 0\quad\text{on}\quad [t_0,1].
\end{align*}
By~\eqref{1306} and~\eqref{1307} the same holds for $h-h_0$ as well, that is,
\begin{align*}
    h-h_0\leq 0\quad\text{on}\quad [0,t_0]\quad\text{and}\quad h-h_0\geq 0\quad\text{on}\quad [t_0,1].
\end{align*}
Thus, \revision{$\tilde H-\tilde H_0$} is non-increasing on $[0,t_0]$ and non-decreasing on  $[t_0,1]$, and since \revision{$\tilde H(0)= \tilde H_0(0), \tilde H(1)=\tilde H_0(1)$}, it follows that \revision{$\tilde H- \tilde H_0$}  is non-positive on $[0,1]$. This implies \revision{a non-strict version of~\eqref{1315}, and since $h$ is different from $h_0$ on a set of positive Lebesgue measure, $\tilde H- \tilde H_0$ does not vanish identically, which together with its  continuity means that the inequality is strict. }
\end{proof}

\begin{proof}[Proof of Theorem \ref{2311}, lower bound in \eqref{1719}]
\revision{A non-strict version is now a direct consequence of Lemma \ref{lem:LB1} -- Lemma \ref{lem:LB3}.  To get a strict lower bound it is enough to show that there is no convex body for which the function $h$ is equal to $h_0$ for almost all directions $u$. Indeed, assume the opposite. Then it follows from~\eqref{eq:hTilde} and \eqref{eq:h} that for almost all directions $\bu$,  the function
\begin{align*}
     g(t):=|K\cap (t\bu+\bu^\perp)|^{1/(d-1)}
\end{align*}
is symmetric on the interval
\begin{align*}
    [a,b]= \{t\in\R^1\,:\,K\cap (t\bu+\bu^\perp)\ne \emptyset\}
\end{align*}
and linear on each of its two halves.
Let
\begin{align*}
    \tilde K:=\conv\left(\{\bx_a\}\cup\{\bx_b\}\cup \left(K\cap \left(\frac{a+b}{2}\bu+\bu^\perp\right)\right)\right),
\end{align*}
where $\bx_a,\bx_b$ are any points satisfying $\bx_a\in K\cap (a\bu+\bu^\perp)$ and $\bx_b\in K\cap (b\bu+\bu^\perp)$. It is straightforward that the function
\begin{align*}
     \tilde g(t):=|\tilde K\cap (t\bu+\bu^\perp)|^{1/(d-1)}
\end{align*}
is also  symmetric on $[a,b]$ and linear on each of its two halves. Thus using that $g,\tilde g$ vanish at $a,b$ and that $g(\frac{a+b}{2})=\tilde g(\frac{a+b}{2})$ we conclude that $g=\tilde g$, which due to Fubini's theorem implies that 
 $|K|=|\tilde K|$, and since  $\tilde K\subset K$ and $K,\tilde K$ are closed, we have that $K=\tilde K$. It means that the orthogonal projection of $K$ onto any $2$-plane passing through $\bu$ is a quadrangle with one of the diagonal \emph{orthogonal} to $\bu$. But of course this cannot hold for almost all $\bu$. 
}
\end{proof}

\subsection{The upper bound I: Existence of maximizers}\label{subsec:UpperBound1}

Our goal in this section and in Sections \ref{subsec:UpperBound2} and \ref{subsec:UpperBound3} below is to maximize the quantity
\begin{align}\label{eq:IntegralI(h)2}
    I(h):=\frac12\int_{\R}\int_{\R}|t_1-t_2|\, h(t_1)\, h(t_2)\,\dd t_2\,\dd t_1,
\end{align}
under the conditions
\begin{enumerate}
\item[(a)] $h\geq 0$;
\item[(b)] $\supp(h)= [-1,1]$;
\item[(c)] $\int_{\R} h(t)\,\dd t=1$;
\item[(d)] $h^{1/(d-1)}$ is concave on its support. 
\end{enumerate}
For this, we proceed in several steps and the strategy can roughly be summarized as follows. First, we shall argue that within the class of functions satisfying (a)-(d) the supremum of the functional $I(\,\cdot\,)$ is in fact attained. Then we show that for a maximizer $h$ the function $h^{1/(d-1)}$ is necessarily affine on its support, from which we eventually obtain the upper bound.

\begin{lemma} \label{lem:0916}
    Fix $0<c<C$, and let $(f_i)_{i\in\N}$ be a sequence of functions satisfying
    \begin{enumerate}
    \item[(a')] $f\revision{_i}\geq 0$;
    \item[(b')] $\supp(f\revision{_i})= [-1,1]$;
    \item[(c')] $f\revision{_i}(x)\leq C$ for all $x\in[-1,1]$ and there exists some $x\revision{_i}\in[-1,1]$ such that $c\leq f\revision{_i}(x\revision{_i})$;
    \item[(d')] $f\revision{_i}$ is concave on its support. 
    \end{enumerate}
    There exists a function $f$ satisfying (a')-(d') and subsequence $(f_{i_j})_{j\in\N}$ such that $f_{i_j} \to f$ in the $L_1$-norm, as $j\to\infty$.
\end{lemma}
\begin{proof}
    \revision{The  functions $f_i$ are  concave  and  take values in $[0,C]$. Hence, they are continuous on the interval $(-1,1)$. Take some $\epsilon > 0$.  By concavity, the Lipshitz constants of the functions $f_i$ on the interval $[-1 +\epsilon,1-\epsilon]$ are uniformly bounded by some $B=B(C,\epsilon)$. Indeed, if we would have $f_i(x)-f_i(y)> B(x-y)$ for some $x > y$ in the interval $[-1 +\epsilon,1-\epsilon]$, then by concavity this would imply that $f_i(-1)$ should become negative (if $B$ is chosen sufficiently large),  which is a contradiction.  Similarly, $f_i(x)-f_i(y)<-B(x-y)$ would imply that $f_i(+1)$  must  be  negative,  again  a  contradiction.   Thus,  the  functions $f_i$ are  equicontinuous.   By  the  theorem  of  Arzela-Ascoli,  there  is  a subsequence converging uniformly on $[-1 +\epsilon,1-\epsilon]$.  Such a sequence exists for every $\epsilon= 1/2,1/3,1/4,\ldots$, so by a diagonal argument there is a subsequence of the $f_i$’s converging uniformly on all intervals $[-1 +\epsilon,1-\epsilon]$,  for  all $\epsilon >0$. 
    
    Let $f$ be the function continuous on its support $[-1,1]$ defined as the pointwise limit of the $f_{i_j}$'s in $(-1,1)$ and extended by continuity at $-1$ and $1$.
    The continuity of $f$ in the interior of its support follows from the uniform limit theorem and the continuity of the $f_{i_j}$'s on any interval $[-1+\epsilon,1-\epsilon]$.
    The fact that we can extend continuously the function on the boundary of its support is possible because the functions are uniformly bounded.
    This construction implies directly that $f$ satisfies properties (a'), (b'), (d') and the first part of (c').
    The fact that the functions are uniformly bounded implies also the $L_1$-convergence.
    
    It remains only to show that $f$ satisfies the second part of (c').
    Assume this is not the case.
    Let $\delta>0$ such that $f\leq c-\delta$.
    Such $\delta$ exists since $f$ reaches its supremum by continuity on the compact $[-1,1]$.
    Observe that we can pick $\epsilon>0$ small enough such that $\|f_i \ind_{[c-\delta,\infty)}\|_{L_1} \geq \epsilon$.
    This follows from the concavity of $f_i$ on $[-1,1]$ and the lower bounds $f_i(x_i)\geq c$ and $f_i\geq 0$.
    Therefore $\|f_i-f\|_{L_1} \geq \|f_i \ind_{[c-\delta,c]} \|_{L_1} \geq \epsilon$ which contradicts the $L_1$-convergence.}
\end{proof}
\revision{Note that in the proof above the claim on the uniform equicontinuity is incorrect for $\epsilon= 0$ because of the counterexample in which the function $f_i$ has slope $i$ on the interval $[-1,-1+1/i]$ (and is constant 1 elsewhere).}

\begin{lemma} \label{lem:0917}
    Let $h$ be a function satisfying the conditions (a)-(d). Then the function $f:=h^{1/(d-1)}$ satisfies the conditions (a')-(d') with the constants $c=2^{-1/(d-1)}$ and $C=2^{\frac{d-2}{d-1}}$.
\end{lemma}
\begin{proof}
    Conditions (a'), (b') and (d') are trivially checked and it remains only to prove that $f$ satisfies $(c')$.
    
    First we show that there exists $x\in[-1,1]$ such that $f(x)\geq c$. Otherwise we would have $h(x)<1/2$ for all $x\in[-1,1]$ and this would contradict (c).
    
    Second, by Hölder's inequality, we have
    \begin{align*}
        \| f \|_1 
        \leq \| \ind\{ \, \cdot \in [-1,1]\} \|_{\frac{d-1}{d-2}} \| f \|_{d-1} 
        = C \| h \|_{1}^{\frac{1}{d-1}} 
        = C ,
    \end{align*}
    where the last equality follows from (c).
    Now let $x\in[-1,1]$ be such that $f(x)$ is maximal. By (a'), (b') and (d') we have that $f$ is greater than the continuous piecewise affine function $g$ which is zero outside the interval $(-1,1)$, affine on both $[-1,x]$ and $[x,1]$, and equals $f(x)$ at $x$.
    In particular
    \begin{align*}
        \| f \|_1 \geq \|g\|_1 = f(x) .
    \end{align*}
    Combining the last two displayed equations gives $f(x) \leq C$.
    This concludes the proof.
\end{proof}

\begin{lemma}
    Within the set of functions satisfying (a)-(d) the suppremum of the functional $I(\,\cdot\,)$ given by \eqref{eq:IntegralI(h)2} is attained.
\end{lemma}
\begin{proof}
    Let $(h_i)_{i\geq 1}$ be a sequence of functions satisfying (a)-(d) such that $\lim_{i\to\infty} I(h_i)$ is the suppremum considered in the statement of the lemma.
    Define $f_i = h_i^{1/(d-1)}$ for each $i\geq 1$.

    By Lemma \ref{lem:0917}, we have that $f_i$ satisfy (a')-(d') for each $i\geq 1$. Therefore by Lemma \ref{lem:0916} there exists a function $f$ satisfying (a')-(d') and a subsequence $(f_{i_j})_{j\geq 1}$ converging to $f$ in the $L_1$-norm.
    It follows that the corresponding subsequence $(h_{i_j})_{j\geq 1}$ converges to $h=f^{d-1}$ with respect to the $L_1$-norm.
    \revision{Observe also that $I$ is a continuous functional (with respect to the $L^1$-norm) on the set of functions satisfying (a)-(d). Indeed, for functions $h$ and $g$
    satisfying (a)-(d) we have that
    \begin{align*}
        |I(h)-I(g)| &\leq {1\over 2}\int_\R\int_\R |t_1-t_2|\,|h(t_1)h(t_2)-g(t_1)g(t_2)|\,\dd t_2\,\dd t_1\\
        &={1\over 2}\int_\R\int_\R |t_1-t_2|\,|(h(t_1)-g(t_1))h(t_2)+(h(t_2)-g(t_2))g(t_1)|\,\dd t_2\,\dd t_1\\
        &\leq 2 \|h-g\|_1 \max( \|h\|_\infty , \|g\|_\infty) 
        \leq 2^{1+\frac{d-2}{d-1}} \|h-g\|_1 ,
    \end{align*}
    where we used the facts that $|t_1-t_2|\leq 2$ (property (b)) and $h$ and $g$ are positive (property (a)) and bounded by $2^{\frac{d-2}{d-1}}$ (Lemma \ref{lem:0917}).}
    Therefore \[\lim_{i\to\infty} I(h_{i}) = \lim_{j\to\infty} I(h_{i_j}) = I(h),\] and the lemma holds.
\end{proof}

\subsection{The upper bound II: Precise form of maximizers}\label{subsec:UpperBound2}

After having seen that maximizers for $I(\,\cdot\,)$ exist, we continue by describing their precise form.

\begin{lemma}
    Assume that $h$ satisfies (a)-(d) and is such that $I(h)$ is maximal.
    Then $h^{1/(d-1)}$ is affine on its support.
\end{lemma}
\begin{proof}
    Let's $h$ be as in the statement of the lemma. Assume that $h^{1/(d-1)}$ is not affine on its support. Combined with property (d), it implies that there exists a point $x\in(-1,1)$ at which $h^{1/(d-1)}$ is \textit{strictly} concave, meaning that for any neighborhood of $x$ of the form $(x',x'')\subset (-1,1)$, the linear interpolation of $h^{1/(d-1)}$ defined by $[x',x''] \ni (\lambda x' + ( 1-\lambda) x'') \mapsto \lambda f(x') + (1-\lambda) f(x'') $ is strictly smaller than $h^{1/(d-1)}$ at $x$.
    
    The spirit of the proof is to modify $h$ locally around $x$ such that after normalisation we find a new function satisfying (a)-(d) and for which the functional $I(\,\cdot\,)$ takes a bigger value (this will be illustrated in Figure \ref{fig:Proof}). This gives us a contradiction and implies that the assumption that $h^{1/(d-1)}$ is not affine cannot be satisfied, and therefore the lemma holds.
    The way we modify $h$ will depend on the value of the inner integral
    \begin{align*}
        J(h,t_1) & :=\int_{\R}|t_1-t_2|\, h(t_2)\,\dd t_2 , \quad t_1 \in [-1,1] ,
    \end{align*}
    of $I(h)$. Roughly speaking, if $J(h,x)$ is large we will add some mass to $h$ around $x$ and, on the contrary, if it is small we will take out some mass in a neighborhood of $x$. The threshold between small and large is fixed to be the expected value $\mathbb{E}[X]$ of $J(h,X)$ if $X$ is a real-valued random variable distributed with respect to the probability density $h$. This value is
    \begin{align*}
        \mathbb{E}[X]=\int_{-1}^1 J(h,t) \, h(t) \,\dd t = 2 I(h).
    \end{align*}
    We compute the derivatives of the function $J(h,t)$:
    \begin{align*}
        \frac{\partial J(h,t)}{\partial t} 
        = \int_{-1}^{t} h(t_2) \,\dd t_2 -\int_{t}^{1} h(t_2) \,\dd t_2
        \, , \quad \text{and} \quad
        \frac{\partial^2 J(h,t)}{\partial t^2} 
        = 2 h(t) .
    \end{align*}
    Thus, since $h$ is non-negative by assumption (a), $J(h,t)$ is a convex function of $t\in[-1,1]$. It is even strictly convex on the open interval $(-1,1)$ because the combination of assumptions (a), (b) and (d) implies that $h$ is positive on $(-1,1)$.
    In particular there exists an interval $[\alpha,\beta] \subset [-1,1]$ such that
    \begin{align} \label{0910}
        J(h,t) 
        \begin{cases}
            > 2 I(h) & \text{if } t \in [-1,1]\setminus [\alpha,\beta] 
            \\< 2 I(h) & \text{if } t \in (\alpha,\beta) \, .
        \end{cases}
    \end{align}
    
    \bigskip
    \textbf{Case $x\in (\alpha,\beta)$:}
    Since $J(h,\cdot)$ is continuous and because of \eqref{0910}, there exist a positive constant $C<2$ and a neighborhood $(x',x'')$ of $x$ such that
    \begin{align} \label{0910d}
        J(h,t) < C I(h) \text{ for any } t\in[x',x''].
    \end{align}
    Note that we can choose $x'$ and $x''$ arbitrarily close to $x$.
    Let $\delta:\mathbb{R}\to\mathbb{R}$ be the positive function with support $[x',x'']$ characterised by the properties that $\delta(x')=\delta(x'')=0$ and $(h-\delta)^{1/(d-1)}$ is affine on the closed interval $[x',x'']$. The function $(h-\delta)^{1/(d-1)}$ restricted to $[x',x'']$ is the affine interpolation described at the beginning of this proof. Outside of the interval $[x',x'']$ it is simply the function $h^{1/(d-1)}$.
    The function $h-\delta$ satisfies properties (a), (b) and (d). 
    We will show that, for its normalized version, we have that
    \begin{align} \label{0910a}
        I \left( \frac{ h-\delta}{ \int h(t)-\delta(t) \,\dd t } \right) 
        = \frac{I(h-\delta)}{ 1 - 2 \int \delta(t)\,\dd t + \left( \int \delta(t)\,\dd t \right)^2  }
    \end{align}
    is strictly bigger than $I(h)$.
    We will use the notation
    \begin{align*}
        I(h_1,h_2):=\frac12\int_{\R}\int_{\R}|t_1-t_2|\, h_1(t_1)\, h_2(t_2)\,\dd t_2\,\dd t_1,
    \end{align*}
    for the symmetric bilinear form for which we have $I(h) = I(h,h)$.
    In particular
    \begin{align} \label{0909}
        I \left( h-\delta \right) 
        = I(h) - 2 I(h,\delta) + I(\delta) 
        > I(h) - 2 I(h,\delta) \, ,
    \end{align}
    and
    \begin{align} \label{0910b}
        2 I (h,\delta)
        = \int_{\R} J(h,t) \, \delta(t) \,\dd t
        < C I(h) \int_{\R} \delta(t) \, \dd t \,,
    \end{align}
    where the last inequality follows from \eqref{0910d} and the fact that the support of $\delta$ is $[x',x'']$.
    Combining the inequalities \eqref{0910b}, \eqref{0909} with the equality \eqref{0910a} yields 
    \begin{align} \label{0910e}
        I \left( \frac{ h-\delta}{ \int h(t)-\delta(t) \,\dd t } \right) 
        > I(h) \frac{  1 - C \int \delta(t)\,\dd t }{ 1 - 2 \int \delta(t)\,\dd t + \left( \int \delta(t)\,\dd t \right)^2  } .
    \end{align}
    Since $x'$ and $x''$ can be chosen arbitrarily close to $x$, we can assume that $\int \delta(t)\,\dd t < 2-C$. The latter inequality implies that the right hand side of \eqref{0910e} is strictly bigger than $I(h)$, and we obtained the desired contradiction.
    
    \bigskip
    \textbf{Case $x\in [-1,1]\setminus [\alpha,\beta]$:} Without loss of generality we assume that $-1<x<\alpha$, since otherwise we could consider the function $h(-\, \cdot)$ instead of $h(\cdot)$.
    This time we modify $h$ by increasing it in the interval $[-1,x]$.
    Let $\epsilon>0$ be arbitrarily small, and let $\delta:\mathbb{R}\to\mathbb{R}$ be the smallest positive function such that $(h+\delta)(-1)^{1/(d-1)}=h(-1)^{1/(d-1)} + \epsilon$ and $h+\delta$ satisfy (a), (b) and (d).
    It can be described as follow\revision{s}: its support is of the form $[-1,x']$ for some $x'\in (-1,x]$ and $(h+\delta)^{1/(d-1)}$ is affine on $[-1,x']$.
    Following analogous steps as in the case $x\in(\alpha,\beta)$, we obtain
    \begin{align} \label{0910f}
        I \left( \frac{ h+\delta}{ \int h(t)+\delta(t) \,\dd t } \right) 
        > I(h) \frac{  1 + C \int \delta(t)\,\dd t }{ 1 + 2 \int \delta(t)\,\dd t + \left( \int \delta(t)\,\dd t \right)^2  } ,
    \end{align}
    where $C>2$ is a constant depending only on $h$ and $x$.
    By choosing $\epsilon$ small enough we can ensure that $\int \delta(t)\,\dd t < C-2$, which implies that the right hand side of \eqref{0910f} is strictly bigger than $I(h)$, and we obtained the desired contradiction.
    \begin{figure}
\begin{center}
    \begin{subfigure}[t]{0.32\textwidth}
        \begin{tikzpicture}[scale=2, every node/.style={scale=1}]
            \draw[->, >=latex] (-1.1,0) -- (1.1,0);
            \draw (-1,-0) -- +(0,-1pt) node[below]{$\scriptstyle -1$};
            \draw (1,-0) -- +(0,-1pt) node[below=]{$\scriptstyle 1 $};
            \draw (-0.1,0) -- +(0,-1pt) node[below]{$\scriptstyle \alpha \vphantom{x'}$};
            \draw (0.85,0) -- +(0,-1pt) node[below=]{$\scriptstyle \beta \vphantom{x'}$};
            \draw[domain=-1:0.5, samples=60] plot (\x,0.5*\x+0.6);
            \draw[domain=0.5:1, samples=60] plot (\x,-\x*\x+\x+0.6);
            \draw (-1,0.1) -- (-1,0);
            \draw (1,0.6) -- (1,0);
            \draw (0.2,0.7) -- (0.7,0.81);
            \draw[dashed,black!40] (0.2,0) -- +(0,-1pt) node[below,black]{$\scriptstyle x'$} -- (0.2,0.7); 
            \draw[dashed,black!40] (0.5,0) -- +(0,-1pt) node[below,black]{$\scriptstyle x \vphantom{x'}$} -- (0.5,0.85); 
            \draw[dashed,black!40] (0.7,0) -- +(0,-1pt) node[below,black]{$\scriptstyle x''$} -- (0.7,0.81);
        \end{tikzpicture}
        \caption{Case $x\in (\alpha,\beta).$}
    \end{subfigure}
    \begin{subfigure}[t]{0.32\textwidth}    
        \begin{tikzpicture}[scale=2, every node/.style={scale=1}]
            \draw[->, >=latex] (-1.1,0) -- (1.1,0);
            \draw (-1,-0) -- +(0,-1pt) node[below]{$\scriptstyle -1$};
            \draw (1,-0) -- +(0,-1pt) node[below=]{$\scriptstyle 1 $};
            \draw (-0.1,0) -- +(0,-1pt) node[below]{$\scriptstyle \alpha \vphantom{x'}$};
            \draw (0.1,0) -- +(0,-1pt) node[below=]{$\scriptstyle \beta  \vphantom{x'}$};
            \draw[domain=-1:1, samples=60] plot (\x,-0.5*\x*\x+0.6);
            \draw (-1,0.1) -- (-1,0);
            \draw (1,0.1) -- (1,0);
            \draw (-1,0) -- (-1,0.2) -- (-0.5,0.475);
            \draw[dashed,black!40] (-0.5,0) -- +(0,-1pt) node[below,black]{$\scriptstyle x'$} -- (-0.5,0.475);
            \draw[dashed,black!40] (-0.3,0) -- +(0,-1pt) node[below,black]{$\scriptstyle x \vphantom{x'}$} -- (-0.3,0.555);
        \end{tikzpicture}
        \caption{Case $x\in [-1,1] \setminus [\alpha,\beta].$}
    \end{subfigure}
    \begin{subfigure}[t]{0.32\textwidth}
        \begin{tikzpicture}[scale=2, every node/.style={scale=1}]
            \draw[->, >=latex] (-1.1,0) -- (1.1,0);
            \draw (-1,-0) -- +(0,-1pt) node[below]{$\scriptstyle -1$};
            \draw (1,-0) -- +(0,-1pt) node[below=]{$\scriptstyle 1 $};
            \draw (-0.2,0) -- +(0,-1pt) node[below]{$\scriptstyle x = \alpha \vphantom{x'}$};
            \draw (0.2,0) -- +(0,-1pt) node[below=]{$\scriptstyle \beta \vphantom{x'} $};
            \draw (-1,0) -- (-1,0.2) -- (-0.2,0.5) -- (0.2,0.4) -- (1,0);
            \draw (-1,0) -- (-1,0.3) -- (-0.2,0.5);
            \draw[dashed,black!40] (-0.2,0) -- +(0,0.5);
        \end{tikzpicture}  
        \caption{Case $x\in \{\alpha,\beta\}.$} 
    \end{subfigure}
\end{center}
\caption{Illustration of the functions $h^{1/(d-1)}$ and its locally modified version $(h\pm\delta)^{1/(d-1)}$.
In Figure (A) $(h-\delta)^{1/(d-1)}$ differs from $h^{1/(d-1)}$ on $(x',x'')$.
In Figure (B) $(h+\delta)^{1/(d-1)}$ differs from $h^{1/(d-1)}$ on $[-1,x')$.
In Figure (C) $(h+\delta)^{1/(d-1)}$ differs from $h^{1/(d-1)}$ on $[-1,\alpha)$.}
\label{fig:Proof}
\end{figure}
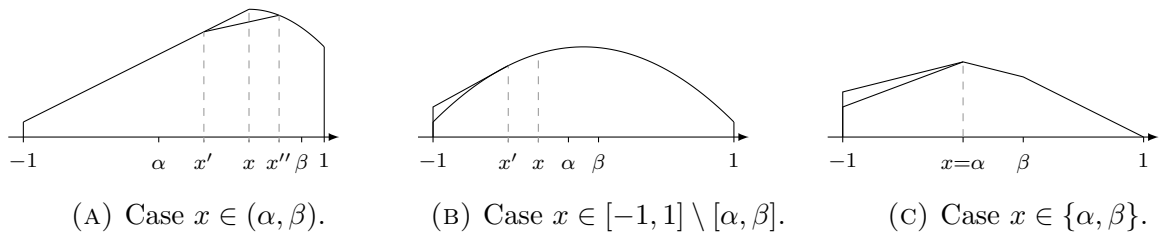

    \bigskip
    \textbf{Case $x\in \{\alpha,\beta\}$:} Without loss of generality we assume that $x = \alpha$, since otherwise we could consider the function $h(-\, \cdot)$ instead of $h(\cdot)$.
    Thanks to the study of the two previous cases we know that $h^{1/(d-1)}$ is affine on each of the three intervals $[-1,\alpha]$, $[\alpha,\beta]$ and $[\beta,1]$.
    We take $\epsilon>0$ arbitrarily small and proceed with the same modification of $h$ as in the previous case. This time we can be a bit more explicit. Since we know that both $h^{1/(d-1)}$ and $(h+\delta)^{1/(d-1)}$ are affine on the interval $[-1,\alpha]$, we can write
    \begin{align*}
        (h+\delta)(t)^{\frac{1}{d-1}} = h(t)^{\frac{1}{d-1}} + \frac{\alpha-t}{\alpha+1}  \epsilon\,, \qquad t\in[-1,\alpha] .
    \end{align*}
    From this we get that there exists a fixed non-negative function $f$ with support $[-1,\alpha]$ such that
    \begin{align*}
        \delta(t)
        = \left( h(t)^{\frac{1}{d-1}} + \frac{\alpha-t}{\alpha+1} \epsilon \right)^{d-1} \!\!\!\! -h(t)
        = (1+o(1)) f(t) \epsilon \,, \qquad t\in[-1,\alpha] .
    \end{align*}
    Moreover the function $f$ is strictly positive on $[-1,\alpha)$.
    This implies that, as $\epsilon\to 0$,
    \begin{align*}
        \frac{\delta(t)}{\int_\R \delta(t) \,\dd t}
        \to \delta_0 (t)
        := \frac{f(t)}{\int_\R f(t) \,\dd t} \,,\qquad t\in[-1,\alpha] .
    \end{align*}
    The function $\delta_0$ is the density of some fixed random variable supported in the interval $[-1,\alpha]$.
    Therefore
    \begin{align*}
        2 I (h,\delta)
        = \int_{\R} J(h,t) \, \delta(t) \,\dd t
        = (1+o(1)) \int_{\R} J(h,t) \, \delta_0(t) \,\dd t \int_\R \delta(t) \,\dd t \,.
    \end{align*}
    Since $\delta_0$ is the density of some random variable supported in $[-1,\alpha]$ and is not concentrated at $\alpha$, we have that $\eqref{0910}$ implies
    \begin{align*}
        C \, I(h) := \int_{\R} J(h,t) \, \delta_0(t) \,\dd t > 2 I(h) .
    \end{align*}
    We can now write 
    \begin{align*} 
        2 I (h,\delta)
        = (1+o(1)) C I(h) \int_\R \delta(t) \,\dd t \,.
    \end{align*}
    This was the most technical part of the proof. Now we finish as in the other cases.
    We have
    \begin{align*}
        I \left( \frac{ h+\delta}{ \int h(t)+\delta(t) \,\dd t } \right) 
        > I(h) \frac{  1 + (1+o(1)) C \int \delta(t)\,\dd t }{ 1 + 2 \int \delta(t)\,\dd t + \left( \int \delta(t)\,\dd t \right)^2  } .
    \end{align*}
    By picking $\epsilon$ sufficiently small the right hand side of the last equation becomes greater than $I(h)$ and we get our contradiction.
\end{proof}

\subsection{The upper bound III: Computation of the maximum}\label{subsec:UpperBound3}

Finally, we are prepared to compute the maximal value the functional $I(\,\cdot\,)$ can attain on the class of functions satisfying (a)-(d).

\begin{lemma}
Assume that $h$ satisfies (a)-(d) and $h^{1/(d-1)}$ is affine on its support. Then for $d\ge 2$ we have
\begin{equation}\label{eq:MaxMax}
    I(h)\leq {1\over 3},
\end{equation}
where the equality holds if and only if $h(t)={1\over 2}\ind_{[-1,1]}(t)$.
\end{lemma}

\begin{proof}
According to the assumptions of the lemma the function $h$ has the following form
$$
h(t)=h_{a,b}(t):=C_{a,b}(at+b)^{d-1}\ind_{[-1,1]}(t),
$$
for some $a,b\in\mathbb{R}$, where due to property (c) we have
$$
C_{a,b}:=\Big(\int_{-1}^1(at+b)^{d-1}\dd t\Big)^{-1}.
$$
Moreover, since for $\tilde h(t):=h(-t)$ we have $I(h)=I(\tilde h)$, without loss of generality we assume $a\ge 0$ and due to property (a) we conclude $b-a\ge 0$.

If $a=0$, then $h_{0,b}(t):=h_0(t)={1\over 2}\ind_{[-1,1]}(t)$ is independent of $b$ and
\begin{equation}\label{eq_11_09_1}
I(h_{0})={1\over 8}\int_{-1}^1\int_{-1}^1|t_1-t_2|\dd t_1\dd t_2={1\over 4}\int_{-1}^1\int_{-1}^{t_2}(t_2-t_1)\dd t_1\dd t_2={1\over 8}\int_{-1}^1(t_2+1)^2\dd t_2={1\over 3}.
\end{equation}

Assume from now on that $a\neq 0$. Then
$$
C_{a,b}^{-1}=\int_{-1}^1(at+b)^{d-1}\dd t={(a+b)^d-(b-a)^d\over da}.
$$
Using the change of variables $s_i=at_i+b$, $i=1,2$ we compute
\begin{align*}
I(h_{a,b})&={C_{a,b}^2\over 2}\int_{-1}^1\int_{-1}^1|t_1-t_2|(at_1+b)^{d-1}(at_2+b)^{d-1}\dd t_1\dd t_2\\
&={C_{a,b}^2\over 2a^3}\int_{b-a}^{b+a}\int_{b-a}^{b+a}|s_1-s_2|s_1^{d-1}s_2^{d-1}\dd s_1\dd s_2\\
&={C_{a,b}^2\over a^3}\int_{b-a}^{b+a}\int_{b-a}^{s_2}(s_2-s_1)s_1^{d-1}s_2^{d-1}\dd s_1\dd s_2.
\end{align*}
We introduce the notation $r_1:=b-a\ge 0$ and $r_2:=a+b>r_1$. Then
\begin{align*}
I(h_{a,b})&={2d^2\over (r_2-r_1)(r_2^d-r_1^d)^2}\int_{r_1}^{r_2}\Big({s_2^{2d}-r_1^ds_2^d\over d}-{s_2^{2d}-r_1^{d+1}s_2^{d-1}\over (d+1)}\Big)\dd s_2\\
&={2d\over (r_2-r_1)(r_2^d-r_1^d)^2}\Big({r_2^{2d+1}-r_1^{2d+1}\over (d+1)(2d+1)}-{r_1^{d}r_2^{d}(r_2-r_1)\over d+1}\Big).
\end{align*}
If $r_1=b-a=0$, then
\begin{equation}\label{eq_11_09_2}
I(h_{a,b})={2d\over (d+1)(2d+1)},
\end{equation}
otherwise let $p:=r_2/r_1-1>0$. With this notation we have
\begin{align*}
I(h_{a,b})&={2d\over (d+1)p((p+1)^d-1)^2}\Big({(p+1)^{2d+1}-1\over 2d+1}-(p+1)^{d}p\Big)=:\ell_d(p).
\end{align*}

In the next step we prove that $\ell_d(q)< {1\over 3}$ for all $d\ge 2$ and $p>0$. We have
\begin{align*}
\ell_d(p)&={2d((p+1)^{2d+1}-1-(2d+1)p(p+1)^d)\over (d+1)(2d+1)p((p+1)^{2d}-2(p+1)^d+1)}\\
&={2d\over (d+1)(2d+1)}{p^3\Big(\sum_{i=3}^{2d+1}{2d+1 \choose i}p^{i-3}-(2d+1)\sum_{i=2}^d{d\choose i}p^{i-2}\Big)\over p^3\Big(\sum_{i=2}^{2d}{2d \choose i}p^{i-2}-2\sum_{i=2}^d{d\choose i}p^{i-2}\Big)}\\
&={2d\over (d+1)(2d+1)}{\sum_{k=0}^{2d-2}{2d+1 \choose k+3}p^{k}-(2d+1)\sum_{k=0}^{d-2}{d\choose k+2}p^{k}\over \sum_{k=0}^{2d-2}{2d \choose k+2}p^{k}-2\sum_{k=0}^{d-2}{d\choose k+2}p^{k}}\\
&={2d\over (d+1)(2d+1)}{ \sum_{k=0}^{d-2}\Big({2d+1 \choose k+3}-(2d+1){d\choose k+2}\Big)p^{k}+\sum_{k=d-1}^{2d-2}{2d+1 \choose k+3}p^{k}\over \sum_{k=0}^{d-2}\Big({2d \choose k+2}-2{d\choose k+2}\Big)p^{k}+\sum_{k=d-1}^{2d-2}{2d \choose k+2}p^{k}}\\
&= {2d\over (d+1)(2d+1)}\,{N_d(p) \over D_d(p)}.
\end{align*}
Consider a polynomial
\begin{align*}
T_d(p)&:= {2d\over (d+1)(2d+1)}N_d(p) - {1\over 3}D_d(p)\\
&={2d\over (d+1)(2d+1)}\Big[\sum_{k=0}^{d-2}\Big({2d+1 \choose k+3}-(2d+1){d\choose k+2}\Big)p^{k}+\sum_{k=d-1}^{2d-2}{2d+1 \choose k+3}p^{k}\Big]\\
&\qquad\qquad-{1\over 3}\Big[\sum_{k=0}^{d-2}\Big({2d \choose k+2}-2{d\choose k+2}\Big)p^{k}+\sum_{k=d-1}^{2d-2}{2d \choose k+2}p^{k}\Big]\\
&=\sum_{k=0}^{d-2}\Big[{2d{2d+1 \choose k+3}\over (d+1)(2d+1)}-{2d{d\choose k+2}\over (d+1)}-{1\over 3}{2d \choose k+2}+{2\over 3}{d\choose k+2}\Big]p^{k}\\
&\qquad\qquad+\sum_{k=d-1}^{2d-2}\Big[{2d{2d+1 \choose k+3}\over (d+1)(2d+1)}-{1\over 3}{2d \choose k+2}\Big]p^{k}.
\end{align*}
Our goal is to show that for $d\ge 2$ all coefficients of polynomial $T_d$ are negative, which would mean that $T_d(p)< 0$ for $p>0$. Note that moreover $T_d(0)=0$.

Consider first the coefficients $\alpha_{k,d}$ for $0\leq k\leq d-2$, where
\begin{align*}
\alpha_{k,d}:&={2d{2d+1 \choose k+3}\over (d+1)(2d+1)}-{2d{d\choose k+2}\over (d+1)}-{1\over 3}{2d \choose k+2}+{2\over 3}{d\choose k+2}\\
&={2d\over 3(k+2)!(d+1)}\Big(\Big({3-k\over k+3}d-1\Big)(2d-1)\ldots(2d-k-1)\\
&\hspace{7cm}-(2d-1)(d-1)\cdots(d-k-1)\Big).
\end{align*}
It is clear, that $\alpha_{k,d}< 0$ for $k\ge 3$ and it is easy to check that $\alpha_{0,d}=\alpha_{1,d}=0$ for any $d\ge 2$. For $k=2$ we have
\begin{align*}
\alpha_{2,d}&={2d\over 3\cdot 4!(d+1)}\Big(\Big({1\over 5}d-1\Big)(2d-1)(2d-2)(2d-3)-(2d-1)(d-1)(d-2)(d-3)\Big)\\
&=-{2d^2(2d-1)(d-1)\over 15\cdot 4!}<0.
\end{align*}
Now consider the coefficients $\beta_{k,d}$ for $d-1\leq k\leq 2d-2$, where
\begin{align*}
\beta_{k,d}:&={2d{2d+1 \choose k+3}\over (d+1)(2d+1)}-{1\over 3}{2d \choose k+2}\\
&={2d(2d-1)\cdots(2d-1-k)\over 3(k+2)!(d+1)}\Big({6d\over k+3}-d-1\Big)\\
&\leq {2d(2d-1)\cdots(2d-1-k)\over 3(k+2)!(d+1)}\Big({6d\over d+2}-d-1\Big)\\
&=-{2d(2d-1)\cdots(2d-1-k)\over 3(k+2)!(d+1)(d+2)}(d^2-3d+2).
\end{align*}
Since the polynomial $d^2-3d+2$ has roots $d=1$ and $d=2$ we conclude that $\beta_{d,k}\leq 0$ for $d-1\leq k\leq 2d-2$ and $d\ge 2$.

Finally we conclude that
\begin{align*}
\ell_d(p)-{1\over 3} &= {2d\over (d+1)(2d+1)}\,{N_d(p)\over D_d(p)}-{1\over 3} = {{2d\over (d+1)(2d+1)}\,N_d(p)-{1\over 3}D_d(p)\over D_d(p)} \\
&={T_d(p)\over \sum_{k=0}^{d-2}\Big({2d \choose k+2}-2{d\choose k+2}\Big)p^{k}+\sum_{k=d-1}^{2d-2}{2d \choose k+2}p^{k}}<0,
\end{align*}
for $p>0$, since ${2d \choose k+2}-2{d\choose k+2}>0$ for $d\ge 2$ and $0\leq k\leq d-2$. 

Combining this with \eqref{eq_11_09_1} and \eqref{eq_11_09_2} and noting that for $d\ge 2$
$$
{2d\over (d+1)(2d+1)}< {1\over 3}
$$
we finish the proof.
\end{proof}

\begin{proof}[Proof of Theorem \ref{2311}, upper bound in \eqref{1719}]
\revision{A non-strict version is now a direct consequence the results we established in Section \ref{subsec:UpperBound1} -- Section \ref{subsec:UpperBound3}. To get a strict upper bound one argue in a very similar way as for the strict lower bound in the proof of Theorem \ref{2311}.}
\end{proof}

\subsection{Sharpness of estimates}\label{1305}
Now let us prove~\eqref{1308}. Due to~\eqref{1719}, it is enough to show that
\begin{align*}
    \limsup_{\delta\to 0}\frac{\Delta(K_\delta)}{V_1(K_\delta)}\leq\frac{3d+1}{2(d+1)(2d+1)}\qquad\text{and}\qquad\liminf_{\delta\to 0}\frac{\Delta(K'_\delta)}{V_1(K'_\delta)}\geq\frac13\revision{.}
\end{align*}
\revision{D}ue to $\lim_{\delta\to 0}V_1(K_\delta)=\lim_{\delta\to 0}V_1(K'_\delta)={2}$\revision{, the above is} equivalent to
\begin{align}\label{2220}
    \limsup_{\delta\to 0}{\Delta(K_\delta)}\leq\frac{3d+1}{(d+1)(2d+1)}\qquad\text{and}\qquad\liminf_{\delta\to 0}{\Delta(K'_\delta)}\geq\frac{{2}}3.
\end{align}
We start with the second part. Let $X_1,X_2$ be independently and uniformly distributed in $K'_\delta$. Denoting by $P_1:\R^d\to\R$ the projection onto the first coordinate, we get
\begin{align*}
    \Delta(K'_\delta)=\E|X_1-X_2|\geq \E|P_1 X_1-P_1 X_2|=\Delta([{-1},1])=\frac{2}{3},
\end{align*}
which implies the \revision{second inequality in}~\eqref{2220}.

Now let $X_1,X_2$ be independently and uniformly distributed in $K_\delta$. We have
\begin{align*}
    \Delta(K_\delta)&=\E|X_1-X_2|\leq \E\left[|P_1 X_1-P_1 X_2|+\delta\right]
    \\\notag
    &=\int_{{-1}}^1\int_{{-1}}^1 |t_1-t_2|\,{h_0(t_1)\,h_0(t_2)}\,\dd t_2\,\dd t_1+\delta
    \\\notag
    &=\frac{3d+1}{(d+1)(2d+1)}+\delta,
\end{align*}
where {$h_0$ is defined in~\eqref{2217} and} in the last step we used~\eqref{2218}. Taking the limit as  $\delta\to 0$, the \revision{first inequality in}~\eqref{2220} follows. This eventually completes the proof of Theorem \ref{2311}.\hfill $\Box$

\subsection*{Acknowledgement}
This project has been iniciated when DZ was visiting Ruhr University Bochum in September and October 2019. Financial support of the German Research Foundation (DFG) via Research Training Group RTG 2131 \textit{High-dimensional Phenomena in Probability -- Fluctuations and Discontinuity} is gratefully acknowledged. We also thank an anonymous referee for insightful comments and remarks which helped us to further improve our paper. We also thank Uwe B\"asel for pointing us to the correct value for $\Delta(H(a))$ in \eqref{eq:Hexagon}.

\bibliographystyle{plain}
\bibliography{bib}

\end{document}